\documentclass[
submission
]{dmtcs-episciences}

\usepackage[utf8]{inputenc}
\usepackage{subfigure}

\newtheorem{theorem}{Theorem}
\newtheorem{corollary}{Corollary}[theorem]
\newtheorem{claim}{Claim}[theorem]
\newtheorem{lemma}[theorem]{Lemma}

\usepackage[round]{natbib}

\author{Hadi Alizadeh\thanks{Corresponding author. Email address: halizadeh@gtu.edu.tr}
  \and Didem Gözüpek}

\title[Upper paired domination versus upper domination]{ Upper paired domination versus upper domination}
\affiliation{
  Department of Computer Engineering, Gebze Technical University, Kocaeli, Turkey}
\received{7$^{th}$Apr. 2021}
\revised{28$^{th}$Oct. 2021}
\accepted{3$^{rd}$Nov. 2021}
\begin{document}
\publicationdetails{23}{2021}{3}{17}{7331}
\maketitle
\begin{abstract}
 A paired dominating set $P$ is a dominating set with the additional property that $P$ has a perfect matching. While the maximum cardinality of a minimal dominating set in a graph $G$ is called the upper domination number of $G$, denoted by $\Gamma(G)$, the maximum cardinality of a minimal paired dominating set in $G$ is called the upper paired domination number of $G$, denoted by $\Gamma_{pr}(G)$. By Henning and Pradhan (2019), we know that $\Gamma_{pr}(G)\leq 2\Gamma(G)$ for any graph $G$ without isolated vertices. We focus on the graphs satisfying the equality $\Gamma_{pr}(G)= 2\Gamma(G)$. We give characterizations for two special graph classes: bipartite and unicyclic graphs with $\Gamma_{pr}(G)= 2\Gamma(G)$ by using the results of Ulatowski (2015). Besides, we study the graphs with $\Gamma_{pr}(G)= 2\Gamma(G)$ and a restricted girth. In this context, we provide two characterizations: one for graphs with $\Gamma_{pr}(G)= 2\Gamma(G)$ and girth at least 6 and the other for $C_3$-free cactus graphs with $\Gamma_{pr}(G)= 2\Gamma(G)$. We also pose the characterization of the general case of $C_3$-free graphs with $\Gamma_{pr}(G)= 2\Gamma(G)$ as an open question.
\end{abstract}

\section{Introduction}
\label{sec:in}

In a graph $G=(V,E)$, a set $D\subseteq V(G)$ is called a \textit {dominating set} if every vertex of $G$ is either in $D$ or adjacent to a vertex in $D$. We say that a dominating set $D$ is \textit {minimal} if no proper subset of $D$ is a dominating set in $G$. The \textit {domination number} of $ G $, denoted by $\gamma(G)$, is the cardinality of a minimum dominating set in $G$, whereas the \textit {upper domination number} of $G$, denoted by $\Gamma(G)$, is the maximum cardinality of a minimal dominating set in $G$.\par
Different variants of domination concept exist in the literature. One of these variants is \textit{paired domination}, which was first put forward by \cite{pdomination}. A \textit{matching} $M$ in a graph $G$ is a set of pairwise non-adjacent edges. If a matching $M$ matches all vertices of a graph $G$, we call $M$ a \textit{perfect matching}. A \textit{paired dominating set (PDS)} of a graph $G$ is a dominating set $D$ of $G$ with the additional property that the subgraph $G[D]$ induced by $D$ contains a perfect matching $M$. In a similar way, the \textit{paired domination number} of a graph $G$, denoted by $\gamma_{pr}(G)$, is the minimum cardinality of a PDS in $G$. In addition, the \textit{upper paired domination number} of a graph $G$, denoted by $\Gamma_{pr}(G)$, is the maximum cardinality of a minimal PDS in $G$.\par

Paired domination is a well-studied subject in the literature. The existing literature on paired domination can be grouped into three major categories. One category includes research works focusing on investigating paired domination number in different graph classes such as trees (\cite{treedomin}), claw-free cubic graphs (\cite{clawfreecubic}), generalized claw-free graphs (\cite{gclawfree}), chordal bipartite graphs (\cite{chordbipartite}), and strongly orderable graphs (\cite{stronglyorder}). Another category contains research works presenting lower bounds for paired domination number (\cite{pdtreelowbound}, \cite{lowerbound2}, \cite{lowerbound1}) and characterization results with respect to a specific relationship between the paired domination number and the graph order (\cite{pairdom23n}, \cite{pairdom35n}, \cite{gamma=n-2}). For further information about lower bounds for paired domination number, the reader is referred to a comprehensive survey by (\cite{Survey}). Furthermore, studies on the relationship between paired domination number and other domination variants such as total domination number (\cite{treetotalpair}, \cite{pairresult}, \cite{totalpair}), induced paired-domination (\cite{inducedpaired}), and double domination (\cite{doubledom}) form the third category.\par

In this paper, we restrict our attention to the concept of upper paired domination, which is a relatively unexplored area of the literature on paired domination. Following the notation by~\cite{Gammapralgo}, we denote by \textit{Upper-PDS}, the problem of finding a $\Gamma_{pr}$-set in a graph $G$. ~\cite{Gammapralgo} studied the concept of upper paired domination from an algorithmic perspective. They showed that while the decision version of \textit{Upper-PDS} problem is NP-complete for general graphs, for some special graph classes including threshold graphs, chain graphs, and proper interval graphs, \textit{Upper-PDS} is solvable in polynomial time.\par
There exist few research works with structural results regarding upper paired domination in the literature. \cite{uptotpr} studied the relationship between the upper paired domination number and the upper total domination number of a graph. They showed that for a graph $G$ with no isolated vertex it holds that $\Gamma_t \geq 1/2(\Gamma_{pr} +2)$. In addition, they gave a characterization for the trees achieving the equality in this relationship. Restricted to the case of connected claw-free graphs $G$, ~\cite{updomclawfree} established upper bounds on $\Gamma_{pr}(G)$ with respect to the graph order $n$ and minimum degree $\delta(G)$. Another available result is due to \cite{implication} where the author provides characterizations for graphs with $\Gamma_{pr}(G)=n$ and $\Gamma_{pr}(G)=n-1$. \par

Due to a result by~\cite{Gammapralgo}, we know that $\Gamma_{pr}(G)\leq 2\Gamma(G)$ for any graph $G$ without isolated vertices. In this paper we focus on graphs which have the property $\Gamma_{pr}(G)= 2\Gamma(G)$. By using the results of \cite{implication}, we give characterizations for two special graph classes: \textit{bipartite} and \textit{unicyclic} graphs with $\Gamma_{pr}(G)= 2\Gamma(G)$. Besides, we study the graphs with $\Gamma_{pr}(G)= 2\Gamma(G)$ with a restricted girth. In this context, we give a complete characterization for graphs with $\Gamma_{pr}(G)= 2\Gamma(G)$ and girth at least 6. Furthermore, for the case of girth at least 4, we characterize $C_3$-free cactus graphs with $\Gamma_{pr}(G)= 2\Gamma(G)$ and leave the characterization of the general case of $C_3$-free graphs with $\Gamma_{pr}(G)= 2\Gamma(G)$ as an open question.  

In Section \ref{prelim}, after introducing some graph-theoretic notations and definitions, we provide some known results in the literature regarding upper paired domination and upper domination. We then proceed to the graphs with $\Gamma_{pr}(G)= 2\Gamma(G)$ in Section \ref{Gpr==2G}, where we focus particularly on bipartite graphs, unicyclic graphs, and graphs with restricted girth.   \par

\section{Preliminaries}\label{prelim}
A \textit{graph} $G$ is an ordered pair $(V (G) , E(G))$, where $V(G)$ is the set of vertices and $E(G)$ is the set of edges each connecting a pair of vertices. Throughout this paper we assume that $G$ is a \textit{simple} graph, that is, a finite, undirected, and loopless graph without multiple edges. The number of vertices of a graph $G$ is called the \textit{order} of $G$. We mean by \textit{neighborhood} of a vertex $v$, denoted by $N(v)$, the set of all vertices that are adjacent to that vertex. The cardinality of $N(v)$ is called the \textit{degree} of vertex $v$ and it is denoted by $deg(v)$. Furthermore, by $\delta(G)$ (resp. $\Delta(G)$ ), we refer to the \textit{minimum} (resp. \textit{maximum}) degree of $G$.\par 

A \textit{subgraph} of a graph $G$ is a graph $H$ such that $V(H)\subseteq V(G)$ and $E(H)\subseteq E(G)$. Furthermore, a subgraph of $G$ \textit{induced} by a set $S\subseteq V(G)$, denoted by $G[S]$, is a graph formed from the vertices of $S$ and all edges connecting the pairs of vertices in $S$. We denote by $P_n$, $C_n$, and $K_n$ a path, a cycle, and a complete graph on $n$ vertices, respectively. The \textit{girth} of a graph is the length of the shortest cycle of that graph. We say that a graph $G$ is \textit{unicyclic} if $G$ is a connected graph containing exactly one cycle. A \textit{cactus} graph is a connected graph in which every edge lies on at most one cycle. \par

The \textit{distance} between two vertices in a graph is the number of edges in a shortest path connecting those vertices. We show by $N_i(v)$ the set of all vertices at distance $i$ from $v$. Notice that with this notation, $N_1(v)$ corresponds to the neighborhood of $v$ which we simply denote by $N(v)$. The \textit{private neighborhood} of a vertex $v\in S$, denoted by $pn(v,S)$, is defined as: $pn(v,S)=\{u \in V(G) \,|\, N(u)\cap S=\{v\}\}$. We call each vertex in $pn(v,S)$ a \textit{private neighbor} of $v$ with respect to set $S$. Furthermore, the \textit{external private neighborhood} of a vertex $v$ with respect to a set $S$, denoted by $epn(v, S)$, is a set containing the private neighbors of $v$ which are not in $S$, that is, $epn(v, S)= pn(v,S)\setminus S$.\par

A \textit{matching} $M$ in a graph $G$ is a set of pairwise non-adjacent edges. If a matching $M$ matches all vertices of a graph $G$, we call $M$ a \textit{perfect matching}. Two vertices are said to be \textit{partners} if they are joined by an edge of a perfect matching $M$.

A set $I$ of vertices in a graph $G$ is an \textit{independent} set if no two vertices in $I$ are adjacent. An independent set is said to be \textit{maximal} if no other independent set properly contains it. The maximum size of an independent set in a graph $G$, denoted by $\alpha(G)$, is called the \textit{independence number} of $G$.\par

Notice that any independent set $S$ in a graph $G$ can be extended to a maximal independent set $I$ in $G$. In addition, every maximal independent set is a minimal dominating set. We will frequently use these two arguments in our forthcoming proofs.  \par

The most related results in the literature which provide useful tools for our work is due to ~\cite{implication}. We state the first result of Ulatowski in Lemma \ref{Gammapr=n}, which describes the graphs with upper domination number equal to their order.

\begin{lemma}\label{Gammapr=n}
(\cite{implication}) For a graph $G$ of order $n$, $\Gamma_{pr}(G)=n$ if and only if $G$ is isomorphic to $mK_2$.
\end{lemma}\par
Here, $mK_2$ denotes a graph with $m\geq1$ copies of disjoint $K_2$. The result in Lemma \ref{Gammapr=n} implies that $K_2$ is the only connected graph satisfying $\Gamma_{pr}(G)=n$. The next result establishes an upper bound for the upper domination number of a connected graph.  

\begin{lemma}\label{Gammaprupbound}
(\cite{implication}) If $G$ is a connected graph of order $n\geq 3$, then $\Gamma_{pr}(G)\leq n-1$.
\end{lemma}\par
In the same work, Ulatowski characterized the graphs satisfying the equality in the bound of Lemma \ref{Gammaprupbound}. However, before stating this result, we recall some definitions and notations. The \textit{subdivided star} $K_{1,t}^*$ is the graph obtained from a star $K_{1,t}$ by subdividing every edge once. Let $K_{1,t}^{*\Delta}$ for $\Delta\geq 0$ be a family of graphs obtained by attaching $\Delta$ triangles to the central vertex of a $K_{1,t}^*$  (see Figure \ref{familyk1tdelta}).\par
 Lemma \ref{Gammapr=n-1} states the second result of Ulatowski regarding the graphs with upper domination number equal to one less than their order. 
\begin{figure}[t]
\centering
\includegraphics[width=0.45\textwidth]{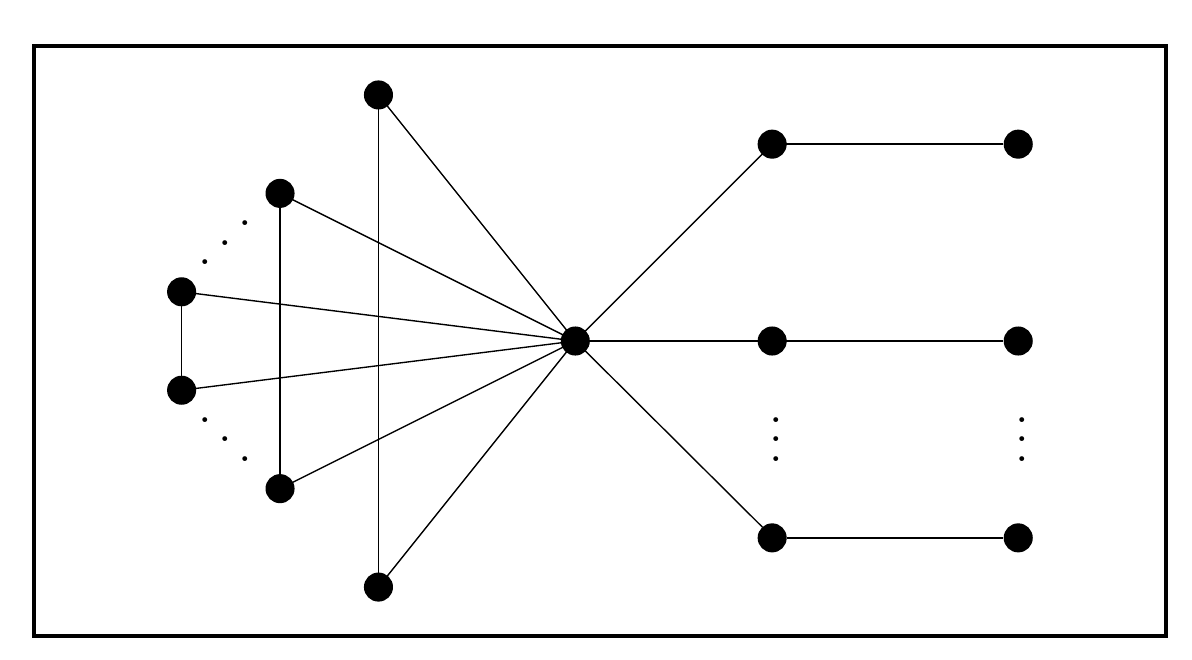}
\caption{A graph in the family $K_{1,t}^{*\Delta}$ }
\label{familyk1tdelta}
\end{figure}

\begin{lemma}\label{Gammapr=n-1}
(\cite{implication}) Let $G$ be a connected graph of order $n\geq 3$. Then $\Gamma_{pr}(G)=n-1$ if and only if $G \in\{C_3, C_5, K_{1,t}^{*\Delta}\}$.
\end{lemma}

The following two lemmas give necessary conditions for the minimality of a paired dominating set. Notice that the notation $epn(u, v; S)$ which is used by ~\cite{updomclawfree} is defined as follows:\\
$\forall u, v \in S$, $epn(u, v; S)= \{w\in N(u)\cup N(v)\setminus S$ $|$ $ N(w)\cap S\subseteq \{u,v\}\}$\\
In other words, for a vertex $w\in epn(u, v; S)$ it holds that either $w\in epn(u,S)$, or $w\in epn(v,S)$, or $w$ is adjacent to both $u$ and $v$ and no other vertex in $S\setminus\{u,v\}$.
  
\begin{lemma}\label{pdsminimality1}
(\cite{updomclawfree}) Let S be a minimal PDS in a connected graph $G$ of order at least 3 and let $\{u, v\} \subset S$ and $S'= S\setminus \{u, v\}$. If $S'$ dominates both $u$ and $v$ and $G[S']$ contains a perfect matching, then $|epn(u, v; S)| \geq 1$.
\end{lemma}\par
  
\begin{lemma}\label{pdsminimality2}
(\cite{updomclawfree}) Let $S$ be a minimal PDS in a connected graph $G$ of order at least 3 and let $M$ be a perfect matching in $G[S]$. If $uv\in M$ and both $u$ and $v$ have degree at least 2 in $G[S]$, then $|epn(u, v; S)| \geq 1$.
\end{lemma}\par

The result in Lemma \ref{relation} is useful in determining the relationship between the upper domination number and the upper paired domination number. 

\begin{lemma}\label{relation}
(~\cite{Gammapralgo}) Every minimal paired dominating set $P$ in $G$ contains a minimal dominating set $S$ such that $|S|\geq |P|/2$. 
\end{lemma}\par

Here we state the relationship between the upper domination number and the upper paired domination number in Corollary \ref{Gpr2G}, which is an immediate result of Lemma \ref{relation}.
\begin{corollary}\label{Gpr2G}
(~\cite{Gammapralgo}) For any graph $G$ without isolated vertices, $\Gamma_{pr}(G)\leq 2\Gamma(G)$.
\end{corollary}

In the remainder of this paper, we will investigate the properties of the graphs satisfying $\Gamma_{pr}(G)= 2\Gamma(G)$.

%%========================================== Section 3 =================================================
\section{Graphs with the property $\Gamma_{pr}(G)= 2\Gamma(G)$}\label{Gpr==2G}
The following result for the graphs with the property $\Gamma_{pr}(G)= 2\Gamma(G)$ is obtained from Lemma \ref{relation}.

\begin{figure}[t]
\centering
\includegraphics[width=0.35\textwidth]{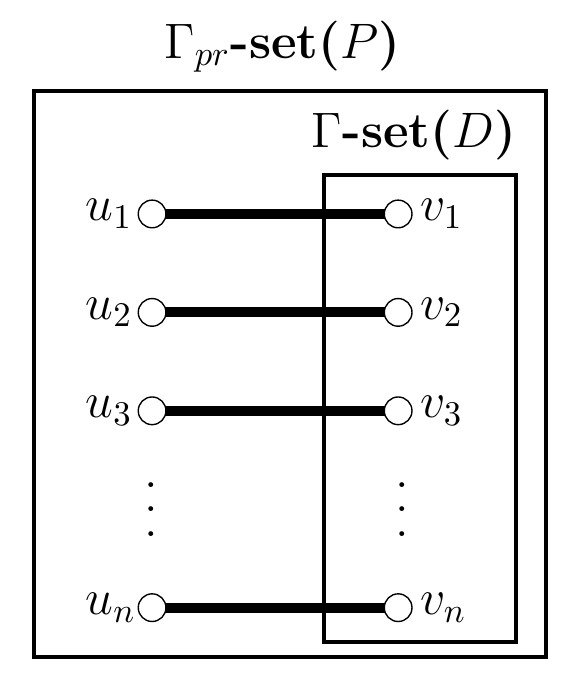}
\caption{A $\Gamma_{pr}$-set $P$ in a graph with the property $\Gamma_{pr}(G)= 2\Gamma(G)$ }
\label{gammas}
\end{figure}

\begin{lemma}\label{independentGamma}
Let $G$ be a graph with the property $\Gamma_{pr}(G)= 2\Gamma(G)$. Then every $\Gamma_{pr}$-set of $G$ contains an independent $\Gamma$-set (see Figure \ref{gammas}). 
\end{lemma}\par
\begin{proof}
Let $P$ be a $\Gamma_{pr}$-set of $G$ with a perfect matching $M$. By Lemma \ref{relation}, $P$ contains a minimal dominating set $D$ such that $|D|\geq\Gamma_{pr}/2$. Since $\Gamma_{pr}(G)= 2\Gamma(G)$, it follows that $|D|\geq \Gamma(G)$. Thus, $D$ is a $\Gamma$-set. Now it suffices to prove that $D$ is an independent set. Suppose to the contrary that $D$ has two adjacent vertices $v_1$ and $v_2$. Suppose further that $u_1v_1, u_2v_2 \in M$. Let $M'=M\setminus \{u_1v_1, u_2v_2\}\cup \{v_1v_2\}$ and $P'= P\setminus \{u_1, u_2\}$. Notice that $P'$ is a dominating set in $G$ since it contains $D$. Furthermore, it has a perfect matching $M'$; therefore, $P'$ is a paired dominating set, a contradiction to the minimality of $P$. Thus, $D$ is an independent set. 
\end{proof}

In the following, we first state our results for two special graph classes with the property $\Gamma_{pr}(G)= 2\Gamma(G)$, namely bipartite and unicyclic graphs. We then investigate special graph classes with $\Gamma_{pr}(G)= 2\Gamma(G)$ and restricted girth.

%%--------------------------------------------------------------- Bipartite graphs with Gamma_{pr}= 2\Gamma$  -------------------------------------------------------------
\subsection{Bipartite graphs with the property $\Gamma_{pr(G)}= 2\Gamma(G)$ }
In this section, we characterize bipartite graphs with $\Gamma_{pr}(G)= 2\Gamma(G)$. We state our obtained result for bipartite graphs with $\Gamma_{pr}(G)= 2\Gamma(G)$ in Theorem \ref{bipartiteGp=2G}.

\begin{theorem}\label{bipartiteGp=2G}
Let $G$ be a connected bipartite graph. Then $\Gamma_{pr}(G)= 2\Gamma(G)$ if and only if $G$ is isomorphic to $K_2$. 
\end{theorem}

\begin{proof}
Let $G$ be a connected bipartite graph with $\Gamma_{pr}(G)= 2\Gamma(G)$ and order $n$. Note that $G$ has at least one partite of size at least $n/2$, which implies a minimal dominating set of size at least $n/2$. Hence $\Gamma(G)\geq n/2$. Then we have $\Gamma_{pr}(G)\geq n$ which implies that $\Gamma_{pr}(G)=n$. By Lemma \ref{Gammapr=n}, $G$ is isomorphic to $mK_2$. Since $G$ is connected, it is isomorphic to $K_2$. \par
For the converse direction, it is easy to see that $\Gamma (K_2)=1$ and $\Gamma_{pr} (K_2)=2$ and hence $\Gamma_{pr} (K_2)= 2\Gamma (K_2)$.
\end{proof}
\par

%%-------------------------------------------------------------- Unicyclic graphs with $\Gamma_{pr}= 2\Gamma$ ---------------------------------------------------------
\subsection{Unicyclic graphs with the property $\Gamma_{pr}(G)= 2\Gamma(G)$ }
The aim of this section is to describe unicyclic graphs with the property $\Gamma_{pr}(G)= 2\Gamma(G)$. Before stating our result on unicyclic graphs with the property $\Gamma_{pr}(G)= 2\Gamma(G)$ in Theorem \ref{unicycGp=2G}, we mention the following lemma which establishes lower bounds for upper domination number in unicyclic graphs:

\begin{lemma}\label{Gammauni}
Let $G$ be a connected unicyclic graph of order $n$. Then the following hold:
\begin{itemize}
	\item For even $n$, $\Gamma(G) \geq n/2$ 
	\item For odd $n$, $\Gamma(G) \geq (n-1)/2$ 
\end{itemize}

\end{lemma}

\begin{proof}
Let $G$ be a connected unicyclic graph of order $n$. Note that $G$ has a single cycle, say $C$. Let $x$ and $y$ be two adjacent vertices of $G$ on $C$. Let further $G'$ be a graph obtained by removing the edge between $x$ and $y$; that is, $V(G')=V(G)$ and $E(G')=E(G)- \{xy\}$. Since $G'$ has no cycles, it is a tree and consequently a bipartite graph. If $n$ is even, then $G'$ has either two partites of size $n/2$ or at least one partite, say $A'$ of size at least $n/2$+1. In the former case, at least one of the partites of size $n/2$ in $G'$ is also an independent set in $G$ and hence $\Gamma(G) \geq n/2$. In the latter case, one possibility is that $x$ and $y$ reside in different partites, in which case $A'$ is also an independent set in $G$ of size at least $n/2$+1. However, the other possibility is that $x$ and $y$ reside in the same partite $A'$, in which case $A'-\{x\}$ is an independent set of size at least $n/2$ in $G$. Both possibilities imply that $\Gamma(G) \geq n/2$ for even $n$.\par
On the other hand, if $n$ is odd, then $G'$ has a partite, say $A'$, of size at least $(n+1)/2$. Here two possibilities exist. One is that $x$ and $y$ are in different partites, in which case $A'$ is also an independent set in $G$ and thus, $\Gamma(G)\geq (n+1)/2$. The other possibility is that $x$ and $y$ reside in $A'$, in which case $A'-\{x\}$ is an independent set of size at least $(n+1)/2-1=(n-1)/2$ in $G$ and hence, $\Gamma(G) \geq (n-1)/2$. Both possibilities yield $\Gamma(G) \geq (n-1)/2$ for odd $n$.
\end{proof}

Now we are ready to state the main result of this section in Theorem \ref{unicycGp=2G}.
\begin{theorem}\label{unicycGp=2G}
Let $G$ be a connected unicyclic graph. Then $\Gamma_{pr}(G)= 2\Gamma(G)$ if and only if $G \in \{C_3, C_5, K_{1,t}^{*1}\}$. 
\end{theorem}

\begin{proof}   
Let $G$ be a connected unicyclic graph with $\Gamma_{pr}(G)= 2\Gamma(G)$ and order $n$. In the case of even $n$, by Lemma \ref{Gammauni}, we have $\Gamma (G)\geq n/2$, which yields $\Gamma_{pr}(G)=n$. Then, by Lemma \ref{Gammapr=n}, $G$ is isomorphic to $mK_2$ which is not unicyclic, contradiction. Thus, the case of even $n$ does not lead to a unicyclic graph with $\Gamma_{pr}(G)= 2\Gamma(G)$. However, for odd $n$, it follows from Lemma \ref{Gammauni} that $\Gamma(G)\geq (n-1)/2$. This, in turn, leads to $\Gamma_{pr}(G)\geq (n-1)$. Since $n$ is odd we have $\Gamma_{pr}(G)=(n-1)$. Then, it follows from Lemma \ref{Gammapr=n-1} that $G$ is isomorphic to $\{C_3, C_5, K_{1,t}^{*\Delta}\}$. Obviously, $C_3$ and $C_5$ are unicyclic graphs and $K_{1,t}^{*1}$ (for $\Delta=1$) is the only unicyclic graph in the family $K_{1,t}^{*\Delta}$. Therefore, $G \in \{C_3, C_5, K_{1,t}^{*1}\}$.\par 
For the converse direction, we show that if $G \in \{C_3, C_5, K_{1,t}^{*1}\}$, then $\Gamma_{pr}(G)= 2\Gamma(G)$. For the case of $C_3$ and $C_5$, it is easy to verify that $\Gamma(C_3)= 1$, $\Gamma_{pr}(C_3)= 2$ and $\Gamma(C_5)= 2$, $\Gamma_{pr}(C_5)= 4$. Furthermore, $\Gamma(K_{1,t}^{*1})= t+1$ and $\Gamma_{pr}(K_{1,t}^{*1})= 2(t+1)$.

\end{proof}

%%------------------------------------------------- Graphs with and restricted girth --------------------------------------------------------------------

\subsection{Graphs with $\Gamma_{pr}(G)= 2\Gamma(G)$ and restricted girth}
In this section, we address the problem of describing graphs with $\Gamma_{pr}(G)= 2\Gamma(G)$ from a girth point of view. We begin with the case of graphs with $\Gamma_{pr}(G)= 2\Gamma(G)$ and girth at least 6.\par
We first give some definitions and notations that we frequently use in the forthcoming proofs. Let $P$ be any $\Gamma_{pr}$-set of a graph $G$ with  $\Gamma_{pr}(G)= 2\Gamma(G)$. By $G[P]$, we refer to the subgraph induced by the set $P$. Furthermore, if a vertex in $P$ is only adjacent to a single vertex in $P$, we name it a \textit{leaf} in $G[P]$.  
\begin{theorem}\label{girth6}
Let $G$ be a graph of girth at least 6. Then $\Gamma_{pr}(G)= 2\Gamma(G)$ if and only if $G$ is isomorphic to $mK_2$ (for $m\geq 1)$. 
\end{theorem}

\begin{proof}
Let $G$ be a graph of girth at least 6 with $\Gamma_{pr}(G)= 2\Gamma(G)$. If $\Gamma_{pr}(G)=n$, by Lemma \ref{Gammapr=n}, $G$ is isomorphic to $mK_2$ (for $m\geq 1$) and we are done.\par
We will now show that the case $\Gamma_{pr}(G)\leq (n-1)$ does not lead to a graph with $\Gamma_{pr}(G)= 2\Gamma(G)$ and complete the proof. Suppose that $\Gamma_{pr}(G)\leq (n-1)$. By Lemma \ref{independentGamma}, $G$ has a $\Gamma_{pr}$-set $P$, which has  an independent $\Gamma$-set $B$ inside it. We further define set $A=P\setminus B$ as the set of partners of the vertices in $B$. Let $A=\{a_i\}$ and $B=\{b_i\}$ for $1\leq i\leq \Gamma(G)$. Note that $P$ has a perfect matching including pairs of matched vertices $(a_i,b_i)$ for $a_i \in A$, $b_i\in B$, and $1\leq i\leq \Gamma(G)$. Since $\Gamma_{pr}(G)\leq (n-1)$, it implies that there exists at least one vertex $x$ in $V(G)\setminus P$. \par
We first show that $G[P]\neq mK_2$ where $m=\Gamma(G)$. Suppose to the contrary that $G[P]=\Gamma(G) K_2$. Note that the vertex $x$ is adjacent to at most one vertex of each pair of matched vertices $(a_i,b_i)$ since the girth is at least 6. Let $Z$ be a set including one vertex from each pair of vertices $(a_i,b_i)$ in $P$ which is not adjacent to $x$. Since $G[P]=\Gamma(G) K_2$, the set $Z$ is an independent set. Thus, $Z\cup x$ forms an independent set of size $\Gamma(G)+1$, a contradiction to $B$ being a $\Gamma$-set. Therefore, $G[P]\neq mK_2$ and without loss of generality, we assume that there exist at least two pairs of matched vertices, say $(a_1,b_1)$ and $(a_2,b_2)$, which have two adjacent endpoints, say $a_1a_2\in E(G)$. Now by Lemma \ref{pdsminimality1}, it holds that $|epn(b_1, b_2; P)| \geq 1$. Let $y$ be a vertex in $epn(b_1, b_2; P)$. Due to the girth restriction, $y$ is not adjacent to both of $b_1$ and $b_2$. Thus, $y$ is adjacent to exactly one of $b_1$ and $b_2$, say $b_1$. Notice that for each pair of matched vertices $(a_i,b_i)$ for $2\leq i\leq \Gamma(G)$, one of the following three cases holds: \\
\textbf{Case 1}: $b_ia_1\notin E(G)$.\\
\textbf{Case 2}: $b_ia_1\in E(G)$ and $a_i$ is a leaf in $G[P]$.\\
\textbf{Case 3}: $b_ia_1\in E(G)$ and $a_i$ is not a leaf in $G[P]$.\\
%$B'= \{b_i\in B$ $| $ $b_ia_1\notin E(G)\}$\\      
%$A'= \{a_i\in A$ $|$ $ b_ia_1\in E(G)$ and $a_i$ is a leaf in $G[P]\}$\\
%$C'=\{c \in V(G)\setminus P$ $|$ $ c\in epn(b_i,P)$, $b_ia_1\in E(G)$, and $deg(a_i)\geq 2$ in $G[P] \}$.\\
Note that in Case 3 a vertex $b_i$ for $2\leq i\leq \Gamma(G)$ has a neighbor $a_1$ which is different from its partner $a_i$. Hence $b_i$ has degree at least two in $G[P]$. Besides, the partner of $b_i$, namely $a_i$, has degree at least two in $G[P]$ since it is not a leaf in $G[P]$. Therefore, it follows by Lemma \ref{pdsminimality2} that $|epn(a_i, b_i; P)| \geq 1$. This in turn implies that there exists at least one vertex $c_i$ in $V(G)\setminus P$ which is a private neighbor of $a_i$ and $b_i$. Due to girth at least 6 restriction, $c_i$ is adjacent to exactly one of $a_i$ and $b_i$. Since $b_i$ is a vertex in the $\Gamma$-set $B$, $c_i$ is only adjacent to $b_i$ (see Figure \ref{girth6ABC}).\par 

Now let us define the three sets $A'$, $B'$, and $C$ as follows: for each pair of matched vertices $(a_i,b_i)$ for $2\leq i\leq \Gamma(G)$ if Case 1 holds, then put $b_i$ in $B'$; if Case 2 holds, then put $a_i$ in $A'$, and if Case 3 holds, then put $c_i$ in $C$. It is easy to see that $B'$ is an independent set since $B' \subseteq  B$. Furthermore, $I=A'\cup C\cup \{y\}$ is an independent set since $I\subset N_2(a_1)$ and the girth of $G$ is at least 6. The vertices in $A'$ are leaves in $G[P]$; that is, they are only adjacent to their partners $b_i$ in $B\setminus B'$. Thus, no vertex in $A'$ is adjacent to a vertex in $B'$. Besides, the vertices in $C$ are private neighbors which are only adjacent to a vertex $b_i$ in $B\setminus B'$. Hence, no vertex in $C$ is adjacent to a vertex in $B'$. By definition no vertex in $B'$ is adjacent to $a_1$ and the vertex $y$ is adjacent only to $b_1$. Hence we have that $\{y,a_1\}\cup A'\cup B'\cup C$ is an independent set. Note that the sets $A'$, $B'$, and $C$ are mutually disjoint sets since for each pair of matched vertices $(a_i,b_i)$ for $2\leq i\leq \Gamma(G)$, exactly one of the three aforementioned cases holds. Thus,  $|A'\cup B'\cup C|=\Gamma(G)-1$. Hence, the set $\{y,a_1\}\cup A'\cup B'\cup C$ is an independent set of size $\Gamma(G)+1$, a contradiction to $B$ being a $\Gamma$-set. Therefore, there exists no graph $G$ with $\Gamma_{pr}(G)\leq (n-1)$, $\Gamma_{pr}(G)= 2\Gamma(G)$, and girth at least 6. Hence, $G$ is isomorphic to $mK_2$ (for $m\geq 1$) and we are done.\par
For the converse direction, it can easily be verified that $\Gamma (mK_2)$= $m$ and $\Gamma_{pr} (mK_2)$=$2m$. Therefore, $\Gamma_{pr} (mK_2)$=$\Gamma (mK_2)$.
\end{proof}

\begin{figure}[t]
\centering
\includegraphics[width=0.6\textwidth]{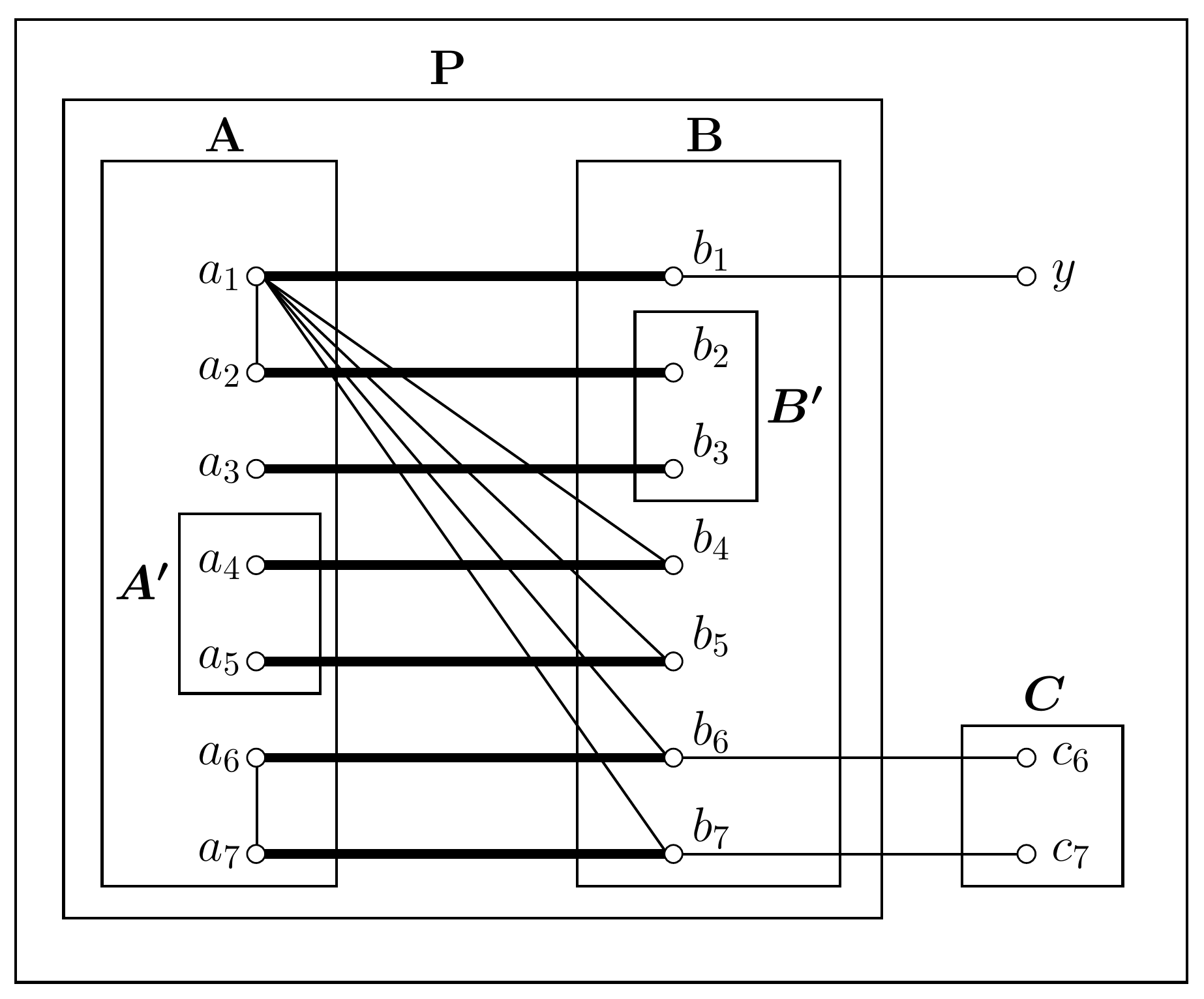}
\caption{The sets $A', B'$, and $C$}
\label{girth6ABC}
\end{figure}

%%------------------------------------------------- Graphs with and restricted girth --------------------------------------------------------------------

In what follows, we proceed to the graphs with $\Gamma_{pr}(G)= 2\Gamma(G)$ and girth smaller than 6. We focus on graphs with $\Gamma_{pr}(G)= 2\Gamma(G)$ and girth at least 4 and provide a characterization for a special family of graphs with the mentioned properties, that is, $C_3$-free cactus graphs with $\Gamma_{pr}(G)= 2\Gamma(G)$.

%We first give a definition for the graph $m_1C_5+m_2K_2$ which is stated in Theorem \ref{girth5result}. The graph $m_1C_5+m_2K_2$ is a graph composed of $m_1$ copies of disjoint $C_5$'s and $m_2$ copies of disjoint $K_2$'s.  
From here onward, we assume that $G$ is a $C_3$-free cactus graph with $\Gamma_{pr}(G)= 2\Gamma(G)$. Furthermore, recall that by Lemma \ref{independentGamma}, $G$ has a $\Gamma_{pr}$-set $P$ with an independent $\Gamma$-set inside it. Let further $P=A\cup B$, where $B$ is an independent $\Gamma$-set and $A$ is the set of partners of the vertices in $B$. Let $A=\{a_i\}$ and $B=\{b_i\}$ for $1\leq i\leq \Gamma(G)$. Note that $P$ has a perfect matching including pairs of matched vertices $(a_i,b_i)$ for $a_i \in A$, $b_i\in B$, and $1\leq i\leq \Gamma(G)$. We now continue with presenting a number of lemmas which provide essential tools for the characterization of $C_3$-free cactus graphs with $\Gamma_{pr}(G)= 2\Gamma(G)$. \par

\begin{lemma}\label{goldrule}
Let $G$ be $C_3$-free cactus graph with $\Gamma_{pr}(G)= 2\Gamma(G)$. Let further $P$ be any $\Gamma_{pr}$-set of $G$. Then, for $1\leq i\leq \Gamma(G)$, at least one vertex of each pair of matched vertices $(a_i,b_i)$ is a leaf in $G[P]$.
\end{lemma}

\begin{proof}
Suppose to the contrary that there exist $k$ pairs of matched vertices $(a_i,b_i)$ in $P$ such that both $a_i$ and $b_i$ have degree at least $2$ in $G[P]$ for $1\leq k\leq \Gamma(G)$. We first look at the case $k=1$, where there exists a single pair of matched vertices, say $(a_1,b_1)$, in $P$ such that both $a_1$ and $b_1$ have degree at least $2$ in $G[P]$. By Lemma \ref{pdsminimality2}, $|epn(a_1, b_1; P)| \geq 1$, which implies that $a_1$ and $b_1$ have a private neighbor $x_1$ in $V(G)\setminus P$. Since $G$ is a $C_3$-free graph, $x_1$ is not adjacent to both $a_1$ and $b_1$. Thus, $x_1$ is adjacent to exactly one of $a_1$ and $b_1$. Since $B$ is a $\Gamma$-set, $x_1$ is adjacent to $b_1$. We define $I_L$ as a set containing one leaf in $G[P]$ from each pair of matched vertices $(a_i,b_i)$ for $2\leq i\leq \Gamma(G)$ in. It is clear that $I_L$ is an independent set. Thus, $\{x_1, a_1\}\cup I_L$ is an independent set of size $\Gamma(G)+1$, which implies a minimal dominating set of size at least $\Gamma(G)+1$, a contradiction to $B$ being a $\Gamma$-set of $G$. Hence, we are done with the case $k=1$.\par
\begin{figure}[t]
\centering
\includegraphics[width=0.9\textwidth]{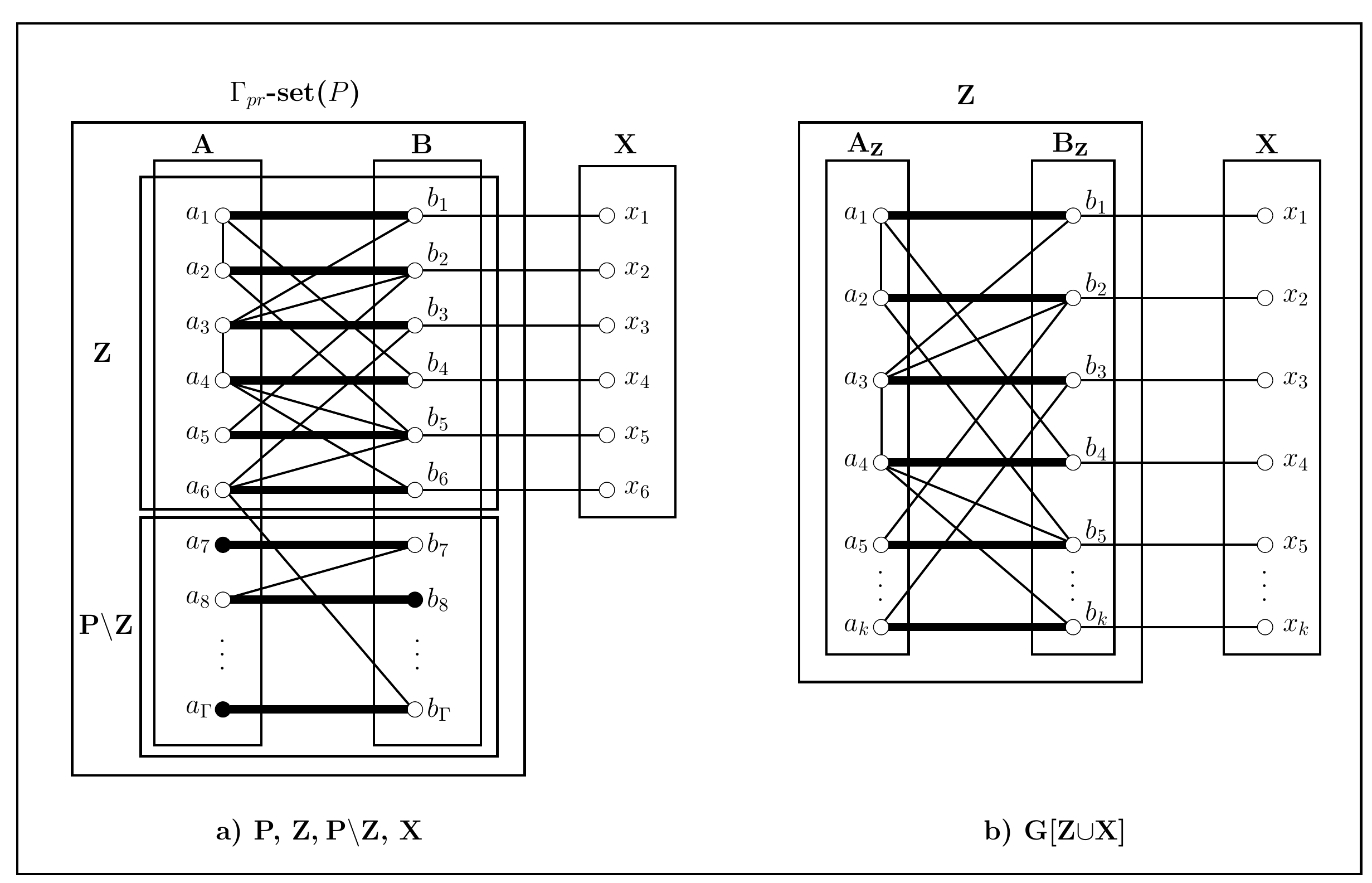}
\caption{The sets $P, Z, P\setminus Z$ and $X$ in $G$ and the subgraph $G[Z\cup X]$ }
\label{inducedZX}
\end{figure}
Then we proceed to the case with $k$ pairs of matched vertices $(a_i,b_i)$ in $P$ such that both $a_i$ and $b_i$ have degree at least $2$ in $G[P]$ for $2\leq k\leq \Gamma(G)$. Let $Z$ be the set containing pairs of matched vertices $(a_i,b_i)$ in $P$ such that both $a_i$ and $b_i$ have degree at least $2$ in $G[P]$. We further assume that $Z=A_z\cup B_z$ where $A_z\subseteq A$ and $B_z\subseteq B$ (see Figure \ref{inducedZX}). By Lemma \ref{pdsminimality2}, for each pair of $(a_i,b_i)$ in $Z$, we have that $|epn(a_i, b_i; P)| \geq 1$ for $1\leq i\leq k $. This implies that each pair of vertices $a_i$ and $b_i$ in $Z$ have a private neighbor $x_i$ in $V(G)\setminus P$. The vertex $x_i$ is not adjacent to both $a_i$ and $b_i$ since $G$ is a $C_3$-free graph. Thus, each $x_i$ is adjacent to exactly one of $a_i$ and $b_i$. Since $B$ is a $\Gamma$-set, $x_i$ is adjacent only to $b_i$. We define $X$ as a set containing $x_i$ for $1\leq i\leq k $ (see Figure \ref{inducedZX}). Notice that from each pair of matched vertices $(a_i,b_i)$ in $P\setminus Z$ at least one vertex is a leaf in $G[P]$. The leaves in $G[P]$ are shown with filled circles in Figure \ref{inducedZX}. Let $I_L$ be a set containing one vertex from each pair of matched vertices $(a_i,b_i)$ in $P\setminus Z$ which is a leaf in $G[P]$. Therefore, $|I_L|= \Gamma(G) - |Z|$. We continue with the following claims.\\
\textbf{Claim 1:} Each $a_i \in A_z$ has at least one neighbor in $B_z$ different from its partner $b_i$.\\
\textbf{Proof of Claim 1}: Suppose to the contrary that a vertex $a_i\in A_z$, say $a_1$, is adjacent only to its partner $b_1$ in $G[Z]$. Then $\{a_1, x_1\}\cup (B_z \setminus \{b_1\})\cup I_L$ is an independent set of size $2 + \Gamma(G) - 1= \Gamma(G)+1$, a contradiction to $B$ being a $\Gamma$-set.\qed\\
\begin{figure}[t]
\centering
\includegraphics[width=0.65\textwidth]{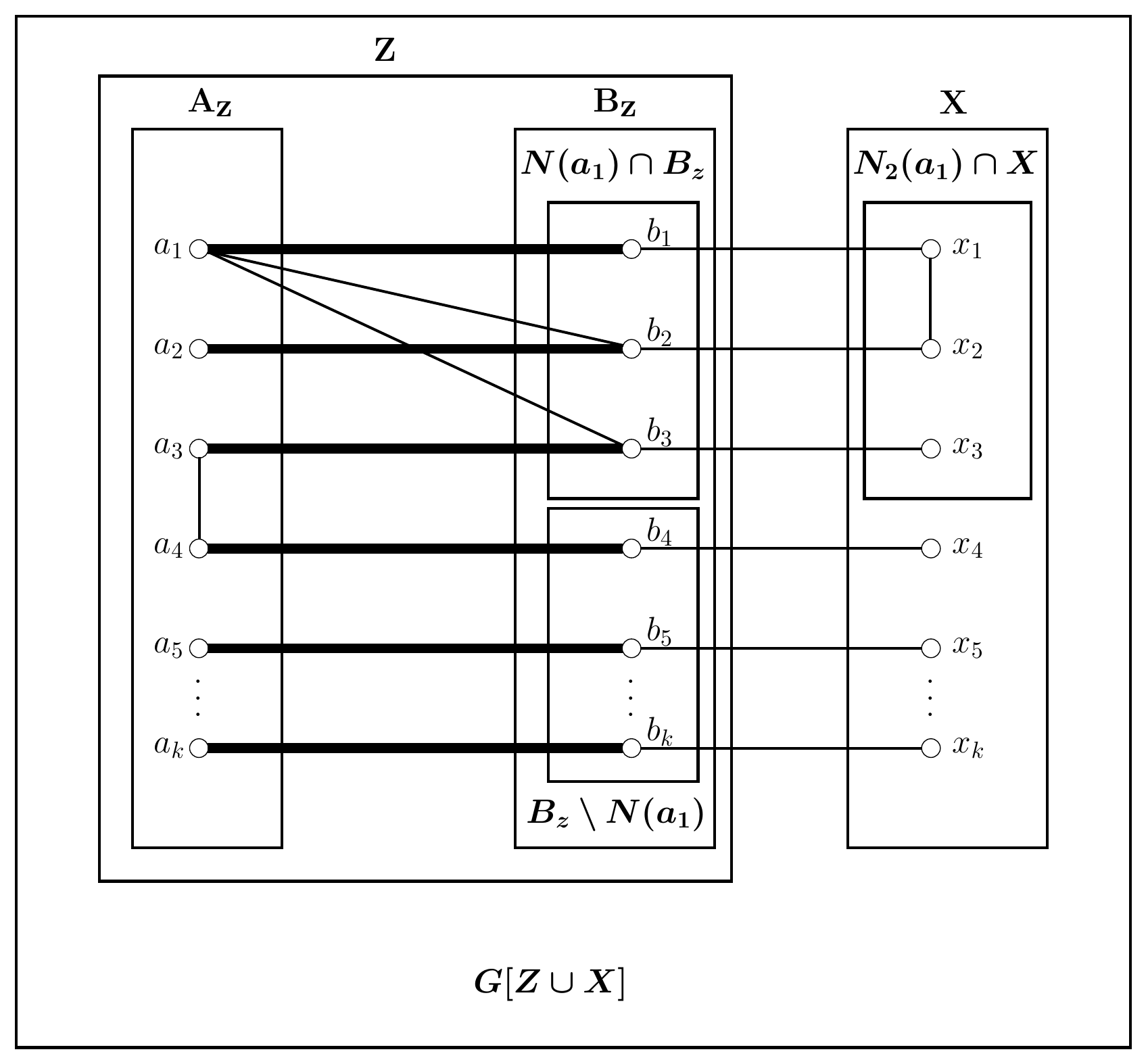}
\caption{The sets $N(a_1)\cap B_z$ and $N_2(a_1)\cap X$ in $G[Z\cup X]$ }
\label{CorNeigh}
\end{figure}

\textbf{Claim 2:} For any $a_i\in A_{Z}$, it holds that at least two vertices in $N_2(a_i)\cap X$ are adjacent.\\
\textbf{Proof of Claim 2:} Suppose to the contrary that there exists a vertex in $A_Z$, say $a_1$, such that $(N_2(a_1)\cap X)$ is an independent set (see Figure \ref{CorNeigh}). Then, $a_1\cup (N_2(a_1)\cap X)\cup (B_Z\setminus N(a_1))\cup I_L$ is an independent set of size $\Gamma(G) +1$, a contradiction to $B$ being a $\Gamma$-set. Thus, at least two vertices in $N_2(a_i)\cap X$ are adjacent.\qed \par
The argument in Claim 2 implies that for each vertex $ a_i \in A_Z$, there is at least one pair of adjacent vertices $(x_k, x_l)$ in $X$. Since $|A_Z|=k$, there must exist at least $k$ pairs of adjacent vertices in $X$. However, since $|X|=k$, there exist at most $k/2$ pairs with disjoint vertices in $X$. Therefore, there exist at least two vertices in $A_z$, say $a_1$ and $a_2$, whose corresponding pairs of adjacent vertices in $X$ are not disjoint; that is, these pairs have either one or two vertices in common. Recall that each vertex $x_i \in X$ is a private neighbor of a vertex $b_i \in B_z$; that is, each $x_i$ is adjacent to a single vertex $b_i$ in $B_z$. Now let $x_1$ and $x_2$ be the corresponding pair of adjacent vertices for $a_1$ in $X$. This implies that $b_2$ is also adjacent to $a_1$ and we have a 5-cycle $C_1=(a_1,b_1,x_1,x_2,b_2)$. Note that if $x_1$ and $x_2$ are also the corresponding pair of adjacent vertices for $a_2$, then $b_1$ is also adjacent to $a_2$ and we have a 5-cycle $C_2=(a_2,b_1,x_1,x_2,b_2)$. However, $C_1$ and $C_2$ are two cycles with a common edge $x_1x_2$, a contradiction to $G$ being a cactus graph. In the other case, if the corresponding pair of adjacent vertices for $a_2$ has only one vertex, say $x_2$, in common with that of $a_1$, then $x_2$ is adjacent to another vertex in $X$, say $x_3$. This in turn implies that $b_3$ is also adjacent to $a_2$ and we have a 5-cycle $C_3=(a_2,b_2,x_2,x_3,b_3)$. In this case, we have two cycles $C_1$ and $C_3$ with a common edge $b_2x_2$, a contradiction to $G$ being a cactus graph. Therefore, there are no pairs of matched vertices $(a_i,b_i)$ in $P$ such that both $a_i$ and $b_i$ have degree at least $2$ in $G[P]$ for $1\leq k\leq \Gamma(G)$.
\end{proof}

Lemma \ref{goldrule} implies that at least one vertex of each pair of matched vertices $(a_i,b_i)$ in $P$ is a leaf  in $G[P]$. We define the set $L_p$ as a set containing one leaf from each pair of matched vertices $(a_i,b_i)$ in $G[P]$ for $1\leq i\leq \Gamma(G)$. It is clear that $L_p$ is an independent set in $G[P]$ and $|L_p|=\Gamma(G)$. In the following lemmas, we obtain some other properties of $C_3$-free cactus graphs with $\Gamma_{pr}(G)= 2\Gamma(G)$. 

\begin{lemma}\label{privateneighbor}
Let $G$ be a $C_3$-free cactus graph with $\Gamma_{pr}(G)= 2\Gamma(G)$. Let further $P$ be any $\Gamma_{pr}$-set of $G$. If there exists a vertex $x$ in $V(G)\setminus P$, it has exactly two neighbors in $P$.
\end{lemma}

\begin{proof}
We first prove that if there exists a vertex $x$ in $V(G)\setminus P$, it has at least two neighbors in $P$.
\begin{claim}\label{rule1}
Every vertex $x$ in $V(G)\setminus P$ has at least two neighbors in $P$.
\end{claim}

\textbf{Proof of Claim \ref{rule1}}: Suppose to the contrary that there exists a vertex $x$ in $V(G)\setminus P$ which has exactly one neighbor in $P$, say $b_1$. By Lemma \ref{goldrule}, one vertex from each pair of matched vertices $(a_i,b_i)$ in $P$ is a leaf in $G[P]$. Let further $L'_p$ be a set containing one leaf in $G[P]$ from each pair of matched vertices $(a_i,b_i)$ for $2\leq i\leq \Gamma(G)$. Thus, $|L'_p|=\Gamma(G)-1$. Then, $\{a_1,x\}\cup L'_p$ is a minimal dominating set of size $2+\Gamma(G)-1=\Gamma(G)+1$, a contradiction to $B$ being a $\Gamma$-set of $G$. Therefore, each vertex $x$ in $V(G)\setminus P$ has at least two neighbors in $P$.\qed\par
Now we proceed by showing that the case of a vertex $x$ in $V(G)\setminus P$ with at least three neighbors in $P$ leads to a contradiction and complete the proof of Lemma \ref{privateneighbor}. Suppose to the contrary that $x$ is a vertex in $V(G)\setminus P$ with at least three neighbors in $P$. We define a set $Z$ as follows: for each pair of matched vertices $(a_i,b_i)$ in $P$, if $a_i\in N(x)$, put $b_i$ in $Z$; otherwise, if $b_i\in N(x)$, put $a_i$ in $Z$.\par 
We first show that $Z$ is an independent set. Suppose to the contrary that two vertices in $Z$, say $a_1$ and $a_2$ are adjacent. By definition of $Z$, the partners of these vertices, namely $b_1$ and $b_2$ are neighbors of $x$. Moreover, since $a_1$ and $a_2$ are adjacent, by Lemma \ref{pdsminimality1}, the vertices $b_1$ and $b_2$ have a private neighbor, say $y$, in $V(G)\setminus P$. Definitely, the vertex $y$ is different from $x$ since $x$ has at least three neighbors in $P$ and cannot be a private neighbor for $b_1$ and $b_2$. However, $y$ is adjacent to exactly one of $b_1$ or $b_2$ since otherwise we have two cycles $(yb_1a_1a_2b_2)$ and $(xb_1a_1a_2b_2)$ with a common edge $a_1a_2$, a contradiction to $G$ being a cactus graph. Thus, $y$ is adjacent to one of $b_1$ or $b_2$, say $b_1$. Then, $y$ is a vertex in $V(G)\setminus P$ with exactly one neighbor $b_1$ in $P$, a contradiction to Claim \ref{rule1}. Therefore, $Z$ is an independent set.\par
Let $L'_p$ be a set containing one leaf in $G[P]$ from each pair of matched vertices $(a_i,b_i)$ in $P$ such that neither $a_i$ nor $b_i$ is adjacent to $x$. It is obvious that $|L'_p|= \Gamma(G) - |Z|$. Then, $\{x\}\cup Z\cup L'_p$ is an independent set of size $\Gamma(G)+1$, which implies a minimal dominating set of size at least $\Gamma(G)+1$, a contradiction to $B$ being a $\Gamma$-set. Therefore, any vertex $x$ in $V(G)\setminus P$ has exactly two neighbors inside $P$.
\end{proof}

\begin{lemma}\label{partneradj}
Let $G$ be $C_3$-free cactus graph with $\Gamma_{pr}(G)= 2\Gamma(G)$. Let further $P$ be any $\Gamma_{pr}$-set of $G$. If there exists a vertex $x$ in $V(G)\setminus P$, then the partners of the two neighbors of $x$ in $P$ are adjacent.
\end{lemma}

\begin{proof}
 Let $x$ be a vertex in $V(G)\setminus P$. By Lemma \ref{privateneighbor}, the vertex $x$ has exactly two neighbors in $P$, say $b_1$ and $b_2$. Suppose to the contrary that the partners of $b_1$ and $b_2$, namely $a_1$ and $a_2$, are non-adjacent. By Lemma \ref{goldrule}, we know that at least one vertex from each pair of matched vertices $(a_i,b_i)$ in $P$ is a leaf in $G[P]$. Let $L'_p$ be the set containing one leaf in $G[P]$ from each pair of matched vertices $(a_i,b_i)$ for $3\leq i\leq \Gamma(G)$. Note that $|L'_p|=\Gamma(G) -2$. Thus, $\{x,a_1,a_2\}\cup L'_p$ is a minimal dominating set of size $\Gamma(G)+1$, a contradiction to $B$ being a $\Gamma$-set. Therefore, the partners of the neighbors of $x$ in $P$, namely $a_1$ and $a_2$, are adjacent.  
\end{proof}

\begin{lemma}\label{comneighinP}
Let $G$  be a $C_3$-free cactus graph with $\Gamma_{pr}(G)= 2\Gamma(G)$. Let further $P$ be any $\Gamma_{pr}$-set of $G$. If there exist two vertices $x_1$ and $x_2$ in $V(G)\setminus P$, then they have no common neighbor in $P$.
\end{lemma}
\begin{proof}  Suppose to the contrary that $x_1$ and $x_2$ are two vertices in $V(G)\setminus P$, which have common neighbors in $P$. By Lemma \ref{privateneighbor}, each of $x_1$ and $x_2$ has exactly two neighbors in $P$. Let $a_1$ and $a_2$ be the two neighbors of $x_1$ in $P$. By Lemma \ref{partneradj}, the partners of $a_1$ and $a_2$, namely $b_1$ and $b_2$, are adjacent. If $x_1$ and $x_2$ have two common neighbors in $P$, that is, if $x_2$ is also adjacent to $a_1$ and $a_2$, then we have two cycles $(x_1a_1b_1b_2a_2)$ and $(x_2a_1b_1b_2a_2)$ with a common edge $b_1b_2$, a contradiction to $G$ being a cactus graph. On the other hand, if $x_1$ and $x_2$ have only one common neighbor, say $a_2$, then $x_2$ has another neighbor in $P$, say $a_3$. By Lemma \ref{partneradj}, the partners of $a_2$ and $a_3$, namely $b_2$ and $b_3$, are adjacent. Then we have two cycles $(x_1a_1b_1b_2a_2)$ and $(x_2a_2b_2b_3a_3)$ with a common edge $a_2b_2$, a contradiction to $G$ being a cactus graph. Hence, $x_1$ and $x_2$ have no common neighbor in $P$.
\end{proof}

\begin{lemma}\label{PDelta2}
Let $G$ be a $C_3$-free cactus graph with $\Gamma_{pr}(G)= 2\Gamma(G)$. Let further $P$ be any $\Gamma_{pr}$-set of $G$. Then, $\Delta(G[P])\leq 2$.
\end{lemma}

\begin{proof} Let $G$ be a $C_3$-free cactus graph with $\Gamma_{pr}(G)= 2\Gamma(G)$. Let further $P$ be any $\Gamma_{pr}$-set of $G$ which includes pairs of matched vertices $(a_i,b_i)$ for $1\leq i\leq \Gamma(G)$. Suppose to the contrary that a vertex in $P$, say $a_1$, has at least three neighbors in $P$. One of these three neighbors is the partner of $a_1$, namely $b_1$. Without loss of generality, let $b_2$ and $b_3$ be the other two neighbors of $a_1$ in $P$. Since $a_1$ is adjacent to $b_2$, by Lemma \ref{pdsminimality1}, we have $|epn(a_2,b_1;P)|\geq 1$, which implies that $a_2$ and $b_1$ have a private neighbor $x$ in $V(G)\setminus P$. By Lemma \ref{privateneighbor}, $x$ is adjacent to both $a_2$ and $b_1$. In addition, since $a_1$ is adjacent to $b_3$, by Lemma \ref{pdsminimality1}, we have $|epn(a_3,b_1;P)|\geq 1$. This implies that $a_3$ and $b_1$ have a private neighbor $y$ in $V(G)\setminus P$. By Lemma \ref{privateneighbor}, $y$ is adjacent to both $a_3$ and $b_1$. However, $x$ and $y$ are two vertices in $V(G)\setminus P$ with a common neighbor $b_1$ in $P$, a contradiction to Lemma \ref{comneighinP}. Thus, a vertex in $P$ has at most two neighbors in $P$; that is, $\Delta(G[P])\leq 2$.
\end{proof}

\begin{lemma}\label{rule7}
Let $G$ be a $C_3$-free cactus graph with $\Gamma_{pr}(G)= 2\Gamma(G)$. Let further $P$ be any $\Gamma_{pr}$-set of $G$. At most one vertex from each pair of matched vertices $(a_i,b_i)$ in $P$ has a neighbor in $V(G)\setminus P$.
\end{lemma}
\begin{proof}
 Suppose to the contrary that there exists a pair of matched vertices in $P$, say $(a_1,b_1)$, such that both $a_1$ and $b_1$ have neighbors in $V(G)\setminus P$. Let further $x_1$ be the neighbor of $b_1$ and $x_2$ be the neighbor of $a_1$ in $V(G)\setminus P$. It is clear that $x_1\neq x_2$ since $G$ is a $C_3$-free graph. By Lemma \ref{privateneighbor}, $x_1$ has two neighbors in $P$. Hence, we may assume that $x_1$ is adjacent to another vertex in $P$, say $b_2$. Similarly, $x_2$ has two neighbors in $P$; however, by Lemma \ref{comneighinP}, $x_2$ has no common neighbor with $x_1$ in $P$. Thus, we may assume that $x_2$ is adjacent to $b_3$. By Lemma \ref{partneradj}, $a_1$ is adjacent to $a_2$ and $b_1$ is adjacent to $a_3$. Then, $(a_1,b_1)$ is a pair of matched vertices both of which have degree at least two in $G[P]$, a contradiction to Lemma \ref{goldrule}. Thus, at most one vertex from each pair of matched vertices $(a_i,b_i)$ in $P$ has a neighbor in $V(G)\setminus P$. 
\end{proof}

\begin{lemma}\label{final}
Let $G$ be a $C_3$-free cactus graph with $\Gamma_{pr}(G)= 2\Gamma(G)$. Let further $P$ be any $\Gamma_{pr}$-set of $G$. Then any two vertices $x_1$ and $x_2$ in $V(G)\setminus P$ are non-adjacent.
\end{lemma}
\begin{proof}Let $G$ be a $C_3$-free cactus graph with $\Gamma_{pr}(G)= 2\Gamma(G)$. Let further $P$ be any $\Gamma_{pr}$-set of $G$ which includes pairs of matched vertices $(a_i,b_i)$ for $1\leq i\leq \Gamma(G)$. Suppose to the contrary that $x_1$ and $x_2$ are two adjacent vertices in $V(G)\setminus P$. We know that by Lemma \ref{privateneighbor} and Lemma \ref{comneighinP}, $x_1$ is adjacent to exactly two vertices in $P$, say $\{b_1,b_2\}$, and $x_2$ is adjacent to two different vertices, say $\{b_3,b_4\}$. By Lemma \ref{partneradj}, the partners of $b_1$ and $b_2$, namely $a_1$ and $a_2$, and the partners of $b_3$ and $b_4$, namely $a_3$ and $a_4$, are adjacent. By Lemma \ref{PDelta2}, $a_1$, $a_2$, $a_3$, and $a_4$ have no other neighbors in $G[P]$. Then there exists an independent set $I= B\setminus \{b_1,b_2,b_3,b_4\}$ in $G[P]$ such that $|I|=\Gamma - 4$. Then $\{x_1, b_1,b_2, a_3,a_4\}\cup I$ is a minimal dominating set of size $5+|I|=5+\Gamma(G) -4= \Gamma(G)+1$, a contradiction to $B$ being a $\Gamma$-set. Therefore, any two vertices $x_1$ and $x_2$ in $V(G)\setminus P$ are non-adjacent.
\end{proof}

Now we are ready to give our main result in this section in Theorem \ref{cactusresult}, which describes the structure of $C_3$-free cactus graphs with $\Gamma_{pr}(G)= 2\Gamma(G)$. Notice that the graph $m_1C_5+m_2K_2$, which is stated in Theorem \ref{cactusresult}, is a graph composed of $m_1$ copies of disjoint $C_5$ and $m_2$ copies of disjoint $K_2$.
%%%%%%%%%%   Main Theorem    %%%%%%%%%%%%%%%%%%%%%%%%%%%%%%%%%%%%%%
\begin{theorem}\label{cactusresult}
Let $G$ be a $C_3$-free cactus graph. Then $\Gamma_{pr}(G)= 2\Gamma(G)$ if and only if $G$ is isomorphic to $m_1K_2+m_2C_5$ for $m_1+m_2\geq 1$.  
\end{theorem}

\begin{proof}
Let $G$ be a $C_3$-free cactus graph with $\Gamma_{pr}(G)= 2\Gamma(G)$. By Lemma \ref{independentGamma}, $G$ has a $\Gamma_{pr}$-set $P$ with an independent $\Gamma$-set $B$ inside it. Let further $A$ be the set of partners of the vertices in $B$. Hence, $P=A\cup B$. Moreover, $P$ has a perfect matching including pairs of matched vertices $(a_i,b_i)$ for $1\leq i\leq \Gamma(G)$. We start with the case where there exist no vertices in $V(G)\setminus P$, that is, $\Gamma_{pr}(G)= n$. By Lemma \ref{Gammapr=n}, $G$ is isomorphic to $m_1K_2$ for $m_1\geq 1$, which is a cactus graph and we are done with this case.\par 
Next, we proceed with the case where there exists at least one vertex $x_1$ in $V(G)\setminus P$, that is, $\Gamma_{pr}(G)\leq n-1$. By Lemma \ref{privateneighbor}, $x_1$ has two neighbors in $P$, say $b_1$ and $b_2$. By Lemma \ref{partneradj}, the partners of $b_1$ and $b_2$, namely $a_1$ and $a_2$, are adjacent. Since $a_1$ and $a_2$ each has two neighbors in $P$, by Lemma \ref{PDelta2}, they have no other neighbors in $P$. By Lemma \ref{rule7}, $a_1$ and $a_2$ have no neighbors in $V(G)\setminus P$ since their partners, namely $b_1$ and $b_2$, have a neighbor $x_1$ in $V(G)\setminus P$. As $a_1$ and $a_2$ each has two neighbors in $G[P]$, by Lemma \ref{goldrule}, their partners, namely $b_1$ and $b_2$, are only adjacent to their partners and have no other neighbors in $G[P]$. Moreover, $b_1$ and $b_2$ have no neighbors in $V(G)\setminus P$ other than $x_1$ by Lemma \ref{comneighinP}. The vertex $x_1$ has two neighbors $b_1$ and $b_2$ in $P$ and has no other neighbors in $V(G)\setminus P$ by Lemma \ref{final}. Hence, the vertices $\{x_1, b_1, a_1, a_2, b_2\}$ form a disjoint 5-cycle in $G$. We can make the previous arguments for any vertex in $V(G)\setminus P$; that is, any vertex in $V(G)\setminus P$ together with four vertices from $P$ form a disjoint 5-cycle in $G$. Therefore, $G$ is composed of components which are either $K_2$ or $C_5$.\par 
For the converse direction, it can easily be verified that if $G$ is isomorphic to $m_1K_2+m_2C_5$ for $m_1+m_2\geq 1$, then we have that $\Gamma(m_1K_2+m_2C_5)=m_1+2m_2$, and $\Gamma_{pr}(m_1K_2+m_2C_5)=2m_1+4m_2$ and hence $\Gamma_{pr}(m_1K_2+m_2C_5)=2\Gamma(m_1K_2+m_2C_5)$.  

\end{proof}

An immediate result of Theorem \ref{cactusresult} for connected graphs is stated in Corollary \ref{congirth5}.
\begin{corollary}\label{congirth5}
Let $G$ be a connected $C_3$-free cactus graph. Then $\Gamma_{pr}(G)= 2\Gamma(G)$ if and only if $G$ is either $C_5$ or $K_2$. 
\end{corollary}\par

Note that some of the arguments used in Lemmas \ref{goldrule}-\ref{final} are not restricted to cactus graphs and can be used for the general case of $C_3$-free graphs. Then the question that arises here is whether all lemmas mentioned above can be extended for the general case of $C_3$-free graphs with $\Gamma_{pr}(G)= 2\Gamma(G)$. Hence, we pose the following as an open question:\\\\
\textbf{Question}: Does Theorem \ref {cactusresult} hold for $C_3$-free graphs?

\newpage

%%%%%%%%%%%%%%%%%%%%%%%%%%%%%%%%%%%%%%%%%%%%%%%%%%%%%%%%%%%%%%%%%%%%%%%%%%%%%%%%%%%%%%%%
\section*{Acknowledgment}
This work is supported by the Scientific and Technological Research Council of Turkey (TUBITAK) under grant no. 118E799. The work of Didem Gözüpek is also supported by the BAGEP Award of the Science Academy of Turkey.

\bibliographystyle{abbrvnat}
% use the following instead if you encounter problems 
%\bibliographystyle{alpha}
\bibliography{sample-dmtcs}
\label{sec:biblio}

\end{document}